\documentclass[11pt]{amsart}
\usepackage[foot]{amsaddr}
\usepackage{mathrsfs}
\usepackage{graphics}
\usepackage{latexsym}
\usepackage{cite}
\usepackage{graphicx,psfrag}
\usepackage{amsmath,amssymb,amsthm}
\newtheorem{Thm}{Theorem}

\newtheorem{Lem}[Thm]{Lemma}
\newtheorem{Prop}[Thm]{Proposition}
\theoremstyle{remark}

\theoremstyle{definition}
\newtheorem{Def}[Thm]{Definition}

\makeatletter
\def\specialsection{\@startsection{section}{1}%
  \z@{\linespacing\@plus\linespacing}{.5\linespacing}%
  {\normalfont}}
\def\section{\@startsection{section}{1}%
  \z@{.7\linespacing\@plus\linespacing}{.5\linespacing}%
  {\normalfont\large\bfseries}}
\makeatother

\begin{document}

\author{Bumtle Kang}
\address{Department of Mathematics Education,
Seoul National University, Seoul 151-742, Korea}
\email{lokbt1@snu.ac.kr}

\author{Suh-Ryung Kim}
\address{Department of Mathematics Education,
Seoul National University, Seoul 151-742, Korea}
\email{srkim@snu.ac.kr}

\author{Boram Park}
\address{Department of Mathematics, Ajou University, Suwon 443-749, Korea}
\email{borampark@ajou.ac.kr}

\title{On the safe set of Cartesian product of two complete graphs}

\begin{abstract}  For a connected graph $G$, a vertex subset $S$ of $V(G)$ is a safe set if for every component $C$ of the subgraph of $G$ induced by $S$, $|C| \ge |D|$ holds for every component $D$ of $G-S$ such that there exists an edge between $C$ and $D$, and, in particular, if the subgraph induced by $S$ is connected, then $S$ is called a connected safe set.
 For a connected graph $G$, the safe number and the connected safe number of $G$ are the minimum among sizes of the safe sets and the minimum among sizes of the connected safe sets, respectively, of $G$.
 Fujita~{\it et al.} introduced these notions in connection with a variation of the facility location problem.
 In this paper, we study the safe number and the connected safe number of Cartesian product of two complete graphs. Figuring out a way to reduce the number of components to two without changing the size of safe set makes it sufficient to consider only partitions of an integer into two parts without which it would be much more complicated to take care of all the partitions.  In this way, we could show that the safe number and the connected safe number of Cartesian product of two complete graphs are equal and present a polynomial-time algorithm to compute them. Especially, in the case where one of complete components has order at most four, we precisely formulate those numbers.

\end{abstract}

\thanks{This work was supported by
the National Research Foundation of Korea(NRF) grant funded by the Korea government(MEST) (No.\ NRF-2015R1A2A2A01006885).}

\maketitle

\section{Introduction}
Fujita~{\it et al.}~\cite{FMS} introduced notions of safe set and connected safe set, motivated by the following problem.
For a given topology of a building, it is required to place temporary accident refuges in addition to business spaces like discussion of conference rooms. Each temporary refuge should be available for the staff in every adjacent business space. (To mitigate the space cost, we assume that each temporary refuge will be used by the people in at most one of the adjacent business space.) Subject to the topology of the building being given, how can the temporary refuges be efficiently located so that the amount of business spaces is maximized? For more recent work on this subject, the reader may refer to Bapat~{\it et al.}~\cite{2}.
 
Given a graph $G$ and a set $X$ of vertices in $G$, we denote by $G[X]$ the subgraph of $G$ induced by $X$. For a connected graph $G$, a set $S$ of vertices in $G$ is said to be a {\em safe set} if for every component $C$ of $G[S]$, $|C| \ge |D|$ holds for every component $D$ of $G-S$ such that there exists an edge between $C$ and $D$, and, especially, if $G[S]$ is connected, then $S$ is called a {\em connected safe set}. For a connected graph $G$, the {\em safe number} $s(G)$ of $G$ is defined as $s(G)=\min\{|S| \mid S \mbox{ is a safe set of }G\}$, and the {\em connected safe number} $cs(G)$ of $G$ is defined as $cs(G)=\min\{|S| \mid S \mbox{ is a connected safe set of }G\}$. See Figure~\ref{fig:safeset} for an illustration. Fujita~{\it et al.}~\cite{FMS} showed that for a graph $G$
\[s(G) \le cs(G) \le 2s(G)-1\]
and any tree $T$ with at most one vertex of degree at least three satisfies the equality $s(T) = cs(T)$. Other than this kind of trees, the complete graphs obviously satisfy the equality. In this regard, we thought that it would be interesting to study which graphs satisfy the equality and Cartesian products of complete graphs are good to start with.

 \begin{figure}
  \begin{center}
\psfrag{a}{\small $v_1$}\psfrag{b}{\small $v_2$}\psfrag{c}{\small $v_3$}
\psfrag{d}{\small $v_4$}\psfrag{e}{\small $v_5$}\psfrag{f}{\small $v_6$}
\psfrag{g}{\small $v_7$}\psfrag{h}{\small $v_8$}\psfrag{i}{\small $v_9$}
\psfrag{j}{\small $v_{10}$}\psfrag{k}{\small $v_{11}$}\psfrag{l}{\small $v_{12}$}
\psfrag{G}{$G$}
  \includegraphics[width=4.5cm]{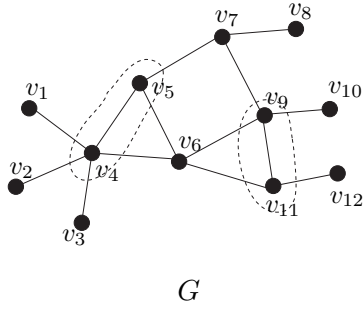}
  \end{center}
  \caption{A set $\{v_4,v_5,v_9,v_{11}\}$ is a safe set. However, $S=\{v_3,v_4,v_9,v_{11}\}$ is not a safe set as $G-S$ has the component with $4$ vertices $v_5,v_6,v_7,v_8$ even if each components of $G[S]$ has size two.} \label{fig:safeset}
\end{figure}

The Cartesian product $G_1 \Box G_2$ of two simple graphs $G_1$ and $G_2$ is a graph with vertex set $V(G_1) \times V(G_2)$ and having two vertices $(u_1,u_2)$ and $(v_1,v_2)$ adjacent if and only if
either $u_1=v_1$ and $u_2$ is adjacent to $v_2$ in $G_2$, or $u_2=v_2$ and $u_1$ is adjacent to $v_1$ in $G_1$.

By figuring out a way to reduce the number of components to two without changing the size of safe set, we shall show that for two integers $m,n \ge 1$, the safe number and the connected safe number of $K_m \Box K_n$ are the same, that is,
\[ s(K_m \Box K_n) = cs(K_m \Box K_n),\]
and go further to compute the exact safe number. By symmetry, we assume $m \le n$ without loss of generality. In addition, we mean by a component $C$ of a graph $G$ both the subgraph $C$ and the vertex set $C$.

\section{Main Results}
 We label the vertices of $K_m$ as $1$, $2$, $\ldots$, $m$ and  $K_n$ as $1$, $2$,$\ldots$, $n$ so that a vertex of $G$ is denoted by $(i,j)$ for some $i\in \{1,2,\ldots,m\}$  and $j\in \{1,2,\ldots,n\}$.

If $m=1$ or $2$, then the safe number and the connected safe number can rather easily be computed:
\begin{Prop}\label{prop:simple}
For any positive integer $n$, the following are true:
\begin{itemize}
\item[(i)] $cs(K_1 \Box K_n)=s(K_1 \Box K_n)=\left\lceil\frac{n}{2}\right\rceil$;
\item[(ii)] $cs(K_2 \Box K_n)=s(K_2 \Box K_n)=n$.
\end{itemize}
\end{Prop}
\begin{proof}
We note that $K_1 \Box K_n$ is the complete graph $K_n$. Thus $K_1 \Box K_n-S$ is connected for a subset $S$ of $V(K_1 \Box K_n)$.  Hence, by the definition of safe set, (i) is immediately true.

Now we show (ii). For simplicity, we let $G=K_2 \Box K_n$. Suppose that $s(G) < n$ and let $S$ be a minimum safe set of $G$. Then $|S| <n$, so $|V(G-S)| > n$. By the Pigeon Hole Principle, there exists $j$, $1 \le j \le n$, such that $(1,j) \in V(G-S)$ and $(2,j) \in V(G-S)$. Since $(1,j)$ and $(2,j)$ are adjacent and each vertex in $G$ is adjacent to $(1,j)$ or $(2,j)$, $G-S$ is connected. However, $|V(G-S)| > |S|$, which contradicts the definition of safe set. Therefore $s(G) \ge n$. Since $\{(1,i) \mid 1 \le i \le n\}$ is a connected safe set of size $n$, we have
\[n \le s(G) \le cs(G) \le n\]
and so (ii) follows.
\end{proof}

From now on, we figure out the safe number and the connected safe number of $K_m \Box K_n$ for $n \ge m \ge 3$.
We denote by $G$ the graph $K_m \Box K_n$ for some integers $n \ge m \ge 3$ throughout this paper.

We first present the following useful proposition.
\begin{Prop}\label{prop:1comp} For $n \ge m \ge 3$, unless $m=n=3$, the following are true:
\begin{itemize}
\item[(i)] There exists a connected safe set of $K_m \Box K_n$ of size $\left\lceil\frac{mn-1}{2} \right\rceil$ that is also a vertex cut;
\item[(ii)] There exists a minimum safe set of $K_m \Box K_n$ that is a vertex cut.
\end{itemize}
\end{Prop}
\begin{proof}
By the division algorithm, $\left\lfloor \frac{mn-1}{2}\right\rfloor = (n-1)q + r$ for some integers $q,r$ with $0 \le r < n-1$. Obviously $q<m$. Since $n-1  \le \left\lfloor \frac{mn-1}{2}\right\rfloor$ for $m \ge 3$, $q \ge 1$. Now we let $C_2$ be the subgraph of $K_m \Box K_n$ induced by \[
 \bigcup_{i=0}^{q-1} \{(m-i,2),\ldots,(m-i,n)\} \cup \{(m-q,n-r+1),\ldots,(m-q,n)\}. \]
 Then $|C_2| = (n-1)q + r= \left\lfloor \frac{mn-1}{2}\right\rfloor$.
  Now we take $(1,1)$ as a trivial subgraph $C_1$ as shown in Figure~\ref{fig:1comp}.
  \begin{figure}
  \begin{center}
    \psfrag{1}{\small $1$}\psfrag{2}{\small $2$}\psfrag{n}{\small $m$}\psfrag{m}{\small $n$}\psfrag{A}{\small $C_1$}
  \psfrag{B}{\small $C_2$}\psfrag{C}{$\looparrowright$}\psfrag{D}{\small $S$}\psfrag{E}{$\left\lfloor \frac{mn-1}{2} \right\rfloor$ vertices}
  \includegraphics[width=4.5cm]{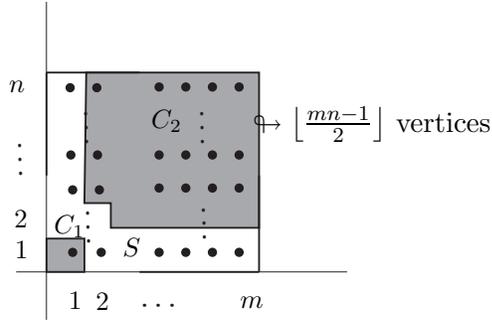}
  \end{center}
  \caption{An illustration of the components of $K_m \Box K_n-S$ mentioned in the proof of Proposition~\ref{prop:1comp}} \label{fig:1comp}
\end{figure}
  Then obviously $C_1$ and $C_2$ are the components of $G-S$ where $S=V(G)-\left(V(C_1) \cup V(C_2)\right)$. Moreover, since $n-r+1 \ge 3$, $(m-q,2) \in S$ and so $G[S]$ is connected. Thus $S$ is a connected safe set. Since $|S| = \left\lceil \frac{mn-1}{2} \right\rceil$, $|C_1| \le |S|$, and $|C_2| \le |S|$. Thus $S$ is a safe set of $K_m \Box K_n$. Hence the safe number of $K_m \Box K_n$ is less than or equal to $\left\lceil \frac{mn-1}{2} \right\rceil$ unless $m=n=3$.

To show (ii), take a minimum safe set $S$ of $K_m \Box K_n$ that is not a vertex cut. Then $G-S$ is connected. Let $C$ be a component of $G[S]$ one of whose vertices is joined to a vertex in $G-S$. Then $mn-|S| \le |V(C)|$. Since $|V(C)| \le |S|$, $mn-|S| \le |S|$, or $|S| \ge \frac{mn}{2}$. Since $|S|$ is an integer, $|S| \ge \left\lceil \frac{mn}{2} \right\rceil$. By (i), there exists a safe set $S^*$ with size $\left\lceil \frac{mn-1}{2} \right\rceil$ that is a vertex cut.
 If $\left\lceil \frac{mn-1}{2} \right\rceil < \left\lceil \frac{mn}{2} \right\rceil$, then $S$ cannot be a minimum safe set since $|S^*|<|S|$. Therefore $\left\lceil \frac{mn-1}{2} \right\rceil = \left\lceil \frac{mn}{2} \right\rceil$.
   Since $|S^*| \le |S|$, $S^*$ is a minimum safe set.  \end{proof}

Let $S$ be a vertex cut of $G$. By definition, any two vertices on the same row or any two vertices on the same column cannot be in distinct components in $G-S$. From this fact, we may make the following simple but very useful observation:
\[\begin{minipage}{0.85\textwidth} If $C$ is a component of $G-S$ and $(i,j)$ is a vertex in $C$, then a vertex in the $i$th column or in the $j$th row belongs to either $C$ or $S$.
\end{minipage} \tag{$\S$}\]
\begin{Def}
Let $C_1,\ldots,C_k$ be the components of $G-S$ for some vertex cut $S$ of $G$.  By the {\it component projection induced by $S$}, we mean the pair $\left(\Pi_1,\Pi_2 \right)$ of functions $\Pi_1:[k]\to 2^{[m]}$ and $\Pi_2:[k] \to 2^{[n]}$ defined as follows: for each $t\in [k]$,
\[
\Pi_1(t)=\{ i \mid (i,j)\in C_t\},\qquad \Pi_2(t)= \{ j \mid (i,j)\in C_t\}.
\]\label{def1}
\end{Def}
\noindent
Since any two vertices on the same row or any two vertices on the same column cannot be in the same component as noted above,
\begin{equation*}
 \Pi_1(s) \cap \Pi_1(t)=\emptyset \mbox{ and } \Pi_2(s) \cap \Pi_2(t)=\emptyset
 \label{disjoint}
 \end{equation*}
  for distinct $s$, $t$ in $[k]$. Thus we may assume that, for any vertex cut $S$ of $K_m \Box K_n$, the components $C_1$, $\ldots$, $C_k$ of $K_m \Box K_n-S$ satisfy
  the following properties throughout this paper:
\[\begin{minipage}{0.85\textwidth} \begin{itemize}
\item[(i)] $(1,1) \in \Pi_1(1) \times \Pi_2(1)$;
\item[(ii)] For $(i_1,j_1) \in C_t$ and $(i_2,j_2) \in C_t'$,  $i_1 < i_2$ and $j_1 < j_2$ if and only if $t<t'$.
     \end{itemize}
\end{minipage} \tag{$\ast$}
\]
 See Figure~\ref{fig:incl} for an illustration.  Moreover, by definition,
\begin{equation}
\left|\bigcup_{t\in[k]}\Pi_1(t)\right|=\sum_{t=1}^{k}  |\Pi_1(t)| \le m,\quad  \left|\bigcup_{t\in[k]}\Pi_2(t)\right|=\sum_{t=1}^{k}  |\Pi_2(t)| \le n.
\label{prop1}
\end{equation}

\begin{figure}
  \begin{center}
    \psfrag{1}{\small $1$}\psfrag{2}{\small $2$}\psfrag{6}{\small $6$}\psfrag{7}{\small $7$}\psfrag{A}{$C_1$}
  \psfrag{B}{$C_2$}\psfrag{C}{$C_3$}
  \includegraphics[width=5cm]{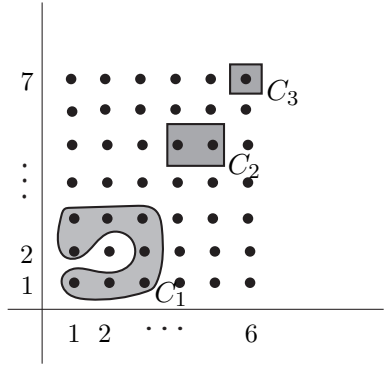}
  \end{center}
  \caption{An illustration of components of $K_6 \Box K_7-S$ arranged so that $(i_1,j_1) \in C_t$ and $(i_2,j_2) \in C'_t$, $i_1<i_2$ and $j_1<j_2$ if and only if $t<t'$ where $S$ is the set of vertices in the unshaded region. For these components, $\Pi_1(1)=\{1,2,3\}$, $\Pi_1(2)=\{4,5\}$, $\Pi_1(3)=\{6\}$, $\Pi_2(1)=\{1,2,3\}$, $\Pi_2(2)=\{5\}$, $\Pi_2(3)=\{7\}$ where $\left(\Pi_1,\Pi_2 \right)$ is the component projection induced by $S$.}\label{fig:incl}
\end{figure}
In this paper, for a vertex cut $S$ of $K_m \Box K_n$, we assume, unless otherwise mentioned, that the components of $K_m \Box K_n$ are arranged in this way.

By ($\ast$), it can easily be checked that, for each $t \in [k]$,
\begin{equation*}
\Pi_1(t) \times \Pi_2(t)=C_t \cup \left(S \cap \left(\Pi_1(t) \times \Pi_2(t)\right)\right)
\end{equation*}
or
\begin{equation}
 C_t= \left(\Pi_1(t) \times \Pi_2(t) \right) \setminus   \left(S \cap \left(\Pi_1(t) \times \Pi_2(t)\right)\right).
 \label{comp}
\end{equation}
For a graph $G$, we denote the number of components of $G$ by $\omega(G)$.

Suppose that $\omega(G-S) \ge 3$. Then there exist two points $(i,j)$ and $(i',j)$ not in $C_2$ for $i$, $i'$, $j$ satisfying $\min \Pi_2(2) \le j \le \max \Pi_2(2)$, $i < \min \Pi_1(2)$, $\max \Pi_1(2) < i'$. Then, by the definition of $K_m \Box K_n$, $(i,j)$ and $(i',j)$ are joined to connect the two regions $R^*$ and $R_*$. See Figure~\ref{fig:S}.
\begin{figure}
\psfrag{1}{\small $1$} \psfrag{2}{\small $2$} \psfrag{m}{$m$}\psfrag{n}{\small $n$}
\psfrag{A}{\small $C_1$}\psfrag{B}{\small $C_2$}\psfrag{C}{\small $C_t$}
\psfrag{a}{\small $R_*$}\psfrag{b}{\small $R^*$}
  \centering
  \includegraphics[width=5cm]{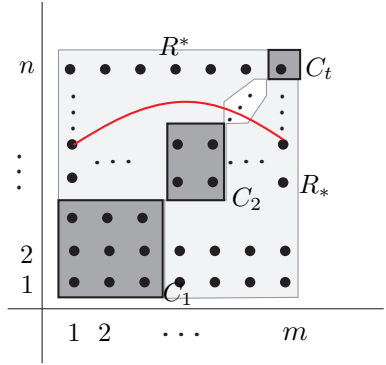}
  \caption{Two points connecting regions $R^*$ and $R_*$ when $t\ge3$}\label{fig:S}
\end{figure}

Therefore, in order for $G-S$ to be disconnected, $t=2$. Now suppose that there exists a point $(i,j)$ not in $C_t$ for some $t \in \{1,2\}$, $\min \Pi_1(1) \le i \le \max \Pi_1(t)$, $\min \Pi_2(t) \le j \le \max \Pi_2(t)$. We assume $t=1$. Then $(i,j)$ is joined to $(i,j')$ and $(i',j)$ for $\max \Pi_1(1) < i'$ and $\max \Pi_2(1) < j'$ to join the two regions $R^*$ and $R_*$. We obtain the same consequence even if $t=2$. Thus we have shown that if $G-S$ is disconnected, then $t=2$, $\Pi_1(1) \times \Pi_2(1)=C_1$, and $\Pi_1(2) \times \Pi_2(2)=C_2$. Therefore we obtain the following lemma:
\begin{Lem}\label{lem_connected}
Let $G = K_m \Box K_n$ for some integers $m,n\ge 1$.
Suppose that one of the following is true for a vertex cut $S$ of $G$:
 \begin{itemize}
 \item[(i)]
either $\sum_{t=1}^{k}  |\Pi_1(t)| < m $ or $\sum_{t=1}^{k} |\Pi_2(t)| < n$ where $k =\omega(G-S)$ and $(\Pi_1,\Pi_2)$ is the component projection induced by $S$;
 \item[(ii)] $\omega(G-S) \ge 3$.
 \end{itemize}
Then the subgraph $G[S]$ is connected.
\end{Lem}
We present a lemma which will play a key role throughout this paper.
\begin{Lem}\label{lem_2comp}
Let $G = K_m \Box K_n$ for some integers $n \ge m\ge 3$ with $n \ge 4$.
 Then there is a minimum safe set $S^*$ of $G$ satisfying $\omega(G-S^*) = 2$.\end{Lem}
\begin{proof} Let $S$ be a minimum safe set of $G$. By Proposition~\ref{prop:1comp}(ii), we may assume that $S$ is a vertex cut. Therefore $\omega(G-S) \ge 2$. Suppose that $k:=\omega(G-S) \ge 3$. Then $G[S]$ is connected by Lemma~\ref{lem_connected}. Let $C_1$, \ldots, $C_k$ be the components of $G-S$ and $(\Pi_1,\Pi_2)$ be the component projection induced by $S$.
Suppose that $|\Pi_1(i)| \le \left\lfloor\frac{m}{2}\right\rfloor$ for all $i \in [k]$. Since $|C_i| \le \Pi_1(i)\Pi_2(i)$ for all $i$, \[
\sum_{i=1}^{k} |C_i| \le \sum_{i=1}^{k}| \Pi_1(i)||\Pi_2(i)|
                 \le \left\lfloor\frac{m}{2}\right\rfloor \sum_{i=1}^{k}|\Pi_2(i)| = \left\lfloor\frac{m}{2}\right\rfloor n \le  \left\lfloor\frac{mn}{2}\right\rfloor.\]
Since $mn = |S| + \sum_{i=1}^{k} |C_i|$, $|S| = mn- \sum_{i=1}^{k} |C_i| \ge mn -  \left\lfloor\frac{mn}{2}\right\rfloor = \left\lceil\frac{mn}{2}\right\rceil$. However, by Proposition~\ref{prop:1comp}(i), there is a safe set of size $\left\lfloor \frac{mn-1}{2} \right\rfloor$. Since $\left\lfloor \frac{mn-1}{2} \right\rfloor < \left\lceil\frac{mn}{2}\right\rceil$, $S$ is not a minimum and we reach a contradiction.  Thus there is $i \in [k]$ such that $\Pi_1(i) > \left\lfloor\frac{m}{2}\right\rfloor$, that is, $\Pi_1(i) \ge \left\lceil\frac{m}{2}\right\rceil$.

Without loss of generality, we may assume that $i=1$. Let $j^*_i$ denote the index of the leftmost column of vertices that belong to $C_i$ for $i=2$, $\ldots$, $k$. Now we form the set $C_2^*$ of vertices in the following way: Take the vertices of $C_2$. Then add a vertex in the $i$th row and the $(j-j^*_l+j^*_2)$th column whenever a vertex in the $i$th row and the $j$th column belongs to $C_l$ for some $l \in \{3,\ldots,k\}$ as shown in  Figure~\ref{fig:2compb}.
\begin{figure}
  \begin{center}
    \psfrag{1}{\small $1$}\psfrag{2}{\small $2$}\psfrag{m}{\small $n$}\psfrag{n}{\small $m$}\psfrag{A}{$C_1$}
  \psfrag{B}{$C_2$}\psfrag{C}{$C_3$}\psfrag{D}{$C_k$}\psfrag{E}{$C_1$}\psfrag{F}{$C_2^*$}
  \includegraphics[width=8cm]{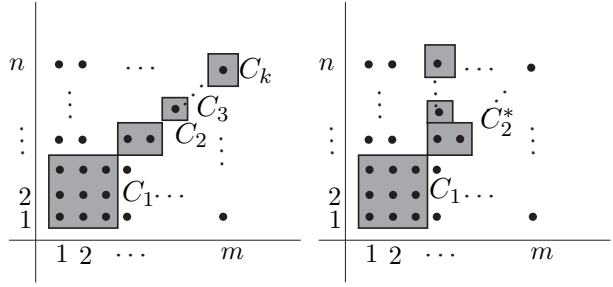}
  \end{center}
  \caption{An illustration of forming $C^*_2$ from $C_2$, $\ldots$, $C_k$}\label{fig:2compb}
\end{figure}
In this way, we obtain a vertex cut $S^*$ of $G$ such that $G-S^*$ consists of two components $C_1$ and $C_2^*$.
Then we can easily check that $|S| = |S^*|$ and $\omega(G-S^*)=2$. Let $(\Pi_1^*,\Pi_2^*)$ be the component projection induced by $S^*$. Then  $\left|\Pi_1^*(1)\right| = \left|\Pi_1(1)\right|$, $\left|\Pi_2^*(1)\right|=\left|\Pi_2(1)\right|$, $\left|\Pi_1^*(2)\right| =\left|\Pi_1(2)\right|$, and $\left|\Pi_2^*(2)\right| =\max_{2 \le j\le k}{\left|\Pi_2(j)\right|}$. Furthermore, since $k \ge 3$, $\left|\Pi_2(1)\right|+\max_{2 \le j\le k}{\left|\Pi_2(j)\right|} < \sum_{i=1}^{k} \left|\Pi_2(i)\right| = n$. Thus  $G[S^*]$ is connected by Lemma~\ref{lem_connected}(i). Hence $S^*$ is the only component the size of which is to be compared with $|C_i|$ for $i=1$, $\ldots$, $k$.

Now, since $S$ is a connected safe set, $|C_1| \le |S|=|S^*|$.
To show that $|C_2^*| \le |S^*|$, recall that $\Pi_1(1) \ge \left\lceil\frac{m}{2}\right\rceil$ and $\Pi_2(1) \ge \left\lceil\frac{n}{2}\right\rceil$ by our assumption.
Thus $\left(m-\left|\Pi_1(1)\right|\right) \le \left\lfloor \frac{m}{2} \right\rfloor \le \left|\Pi_1(1)\right|$ and therefore
\begin{align*}
|C_2^*|   &  \le \left(m-\left|\Pi_1(1)\right|\right)\left(\max_{2 \le j\le k}{\left|\Pi_2(j)\right|} \right)
            \le \left(m-\left|\Pi_1(1)\right|\right)\left(n-\Pi_2(1) \right) \\
            &\le  \left|\Pi_1(1)\right|\left(n-\left|\Pi_2(1)\right|\right) \le |S^*|.\end{align*}
Hence $S^*$ is a connected safe set of size $|S|$.
\end{proof}
\begin{Def} Given integers $m \ge n>1$,
we let
\[P_2(m,n)=\{\left((m_1,m_2),(n_1,n_2)\right) \mid m_1+m_2=m, n_1+n_2=n, m_i,n_i \in \mathbb{N}\},\]
where $\mathbb{N}$ is the set of positive integers.
\label{def:partition}
\end{Def}
\begin{Lem}\label{component_size}
For any $((m_1,m_2),(n_1,n_2)) \in P_2(m,n)$, there is at most one $j\in \{1,2\}$ which satisfies \[
mn-m_1n_1-m_2n_2 < m_jn_j.\]
\end{Lem}

\begin{proof} If $mn-m_1n_1-m_2n_2 \ge m_tn_t$ for all $t \in \{1,2\}$, then we are done.
Suppose that there is $j \in \{1,2\}$ such that $mn-m_1n_1-m_2n_2 < m_jn_j.$
Without loss of generality, we may assume $j=1$, that is,
\[mn-m_1n_1-m_2n_2 < m_1n_1\]
or
\[(m_1+m_2)(n_1+n_2)-m_1n_1-m_2n_2 < m_1n_1.\]
We simplify the above inequality to obtain
\[m_1n_2+(m_2-m_1)n_1<0.\]
Since $m_1n_2>0$ and $n_1>0$, we have $m_1>m_2$.
Now
\[mn-m_1n_1-m_2n_2-m_2n_2=(m_1+m_2)(n_1+n_2)-m_1n_1-m_2n_2-m_2n_2=(m_1-m_2)n_2+m_2n_1.\]
Since $m_1>m_2$, the right hand side of the second equality is positive. Therefore
\[mn-m_1n_1-m_2n_2>m_2n_2.\]
\end{proof}

\begin{figure}
\psfrag{A}{$C_1$} \psfrag{B}{$C_2$} \psfrag{C}{$S_2$} \psfrag{D}{$S_1$}  \psfrag{1}{\small $1$} \psfrag{2}{\small $2$} \psfrag{m}{\small $n$} \psfrag{n}{\small $m$}\psfrag{V}{$V$}
\begin{center}
\includegraphics[height=4.5cm]{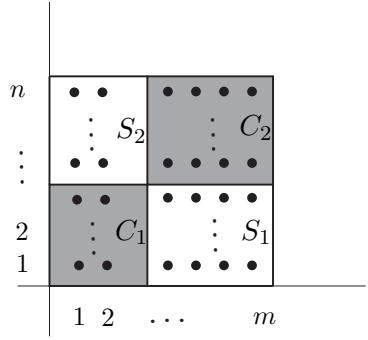}
\caption{$C_1$ and $C_2$} \label{fig:simple}
\end{center}
\end{figure}

In Figure~\ref{fig:simple}, suppose that $V(C_1)=[m_1] \times [n_1]$ and $V(C_2) = \left([m]\setminus[m_1]\right)\times \left([n]\setminus[n_1]\right)$.
Then the subgraph induced by $S:= S_1 \cup S_2$ is not connected.
By taking some vertices in $C_1$ or $C_2$ and adding them to $S$, we would like to obtain a connected safe set $S^*$. We denote the set of such vertices by $\Delta$. Then $|S^*|=mn-m_1n_1-m_2n_2 +|\Delta|$.
If $mn-m_1n_1-m_2n_2\ge \max\{m_1n_1,m_2n_2\}$, then we add just one vertex as we wish to have $S^*$ as small as possible. Otherwise, as long as $S^*$ has at least $\max\{m_1n_1,m_2n_2\}-|\Delta|$, $S^*$ is a safe set. That is, as long as $|S^*|=mn-m_1n_1-m_2n_2 +|\Delta| \ge \max\{m_1n_1,m_2n_2\}-|\Delta|$,
$S^*$ is a safe set.
Solving this inequality for $|\Delta|$ gives
\[|\Delta|\ge \frac{\max\{m_1n_1,m_2n_2\} - mn+\sum_{i=1}^2m_in_i}{2}.\]
Motivated by this observation, we introduce the following notion.

\begin{Def} Given integers $m,n \ge 3$, we define
\begin{align*}
\alpha(m,n) &:= \min\left\{mn-\sum_{i=1}^2m_in_i +\max\left\{\left\lceil\frac{\max\{m_1n_1,m_2n_2\} - mn+\sum_{i=1}^2m_in_i}{2}\right\rceil ,1 \right\} \right\}
\end{align*}
where the minimum is taken for each $((m_1,m_2),(n_1,n_2)) \in P_2(m,n)$.
\end{Def}

Then the following is true.
\begin{Thm}\label{thm_min_connected}
Let $G = K_m \Box K_n$ for some integers $m,n\ge 3$.
Then \[cs(G) = \alpha(m,n).\]
\end{Thm}

\begin{proof}
Let $S$ be a  minimum connected safe set of $G$. By Lemma~\ref{lem_2comp}, we may assume that $\omega(G-S)=2$. Let $C_1$ and $C_2$ be components of $G-S$.
By (\ref{prop1}),
\[  |\Pi_1(1)| +   |\Pi_1(2)| \le m,\quad   |\Pi_2(1)|+ |\Pi_2(2)| \le n\]
where $(\Pi_1,\Pi_2)$ is the component projection induced by $S$.
For notational convenience, let $|\Pi_1(1)|=m_1$, $m-|\Pi_1(1)|=m_2$, $|\Pi_2(1)|=n_1$, and $n-|\Pi_2(1)| = n_2$. Then
\[\left(\left(m_1,m_2\right),\left(n_1,n_2\right) \right) \in P_2(m,n).\]
In addition, we define $r(t)$ in the following way:
\[r(1)=\left|S \cap \left(\Pi_1(1) \times \Pi_2(1)\right)\right| \mbox{ and } r(2)=\left|S \cap \big[\left([m]\setminus\Pi_1(1)\right)\times \left([n]\setminus\Pi_2(1)\right)\big]\right|.\]
Then, by (\ref{comp}), $|C_1|=m_1n_1-r(1)$. Furthermore,
\begin{align*}
C_2&= \big[\left([m]\setminus \Pi_1(1)\right) \times \left([n]\setminus \Pi_2(1)\right)\big]\setminus S \\ &= \big[\left([m]\setminus \Pi_1(1)\right) \times \left([n]\setminus \Pi_2(1)\right)\big]\setminus \big[\left([m]\setminus\Pi_1(1)\right)\times \left([n]\setminus\Pi_2(1)\right)\big], 
\end{align*}
so the equality $|C_2|=m_2n_2-r(2)$ also holds.
If $r(1) = r(2) = 0$, then $S$ is disconnected by Lemma~\ref{lem_connected}. Therefore one of $r(1)$ and $r(2)$ is at least $1$.
Note that if $\sum_{t=1}^{2}  |\Pi_1(t)| = m$ and  $\sum_{t=1}^{2}  |\Pi_2(t)| = n$, then $[n]\setminus  \Pi_2(1)=\Pi_1(2)$ and $[m]\setminus \Pi_1(1)=\Pi_2(2)$.
Thus
\begin{align}
|S|&=  mn-\sum_{t=1}^{2}  |C_t|=mn-\left(m_1n_1-r(1)\right)-\left(m_2n_2-r(2)\right) \notag \\
&=mn-m_1n_1-m_2n_2+r(1)+r(2). \label{eq:S}
\end{align}
By the definition of a connected safe set,
\[m_1n_1-r(1)=|C_1| \le |S| \quad \text{and} \quad
m_2n_2-r(2)=|C_2|\le |S|.\]
Therefore
\[
m_1n_1-r(1)\le mn-m_1n_1-m_2n_2+r(1)+r(2)\]
and\[
m_2n_2-r(2)  \le  mn-m_1n_1-m_2n_2  +r(1)+r(2).\]
Then\[
m_1n_1 - \left( mn-m_1n_1-m_2n_2\right)\le 2r(1)+r(2) \le 2(r(1)+r(2))\]
and\[
m_2n_2 - \left( mn-m_1n_1-m_2n_2\right) \le r(1)+2r(2) \le 2(r(1)+r(2)).\]
Therefore\[
\frac{1}{2}\left(\max\{m_1n_1,m_2n_2\} -\left( mn-m_1n_1-m_2n_2\right)\right) \le r(1)+r(2).\]
Since $r(1)$ and $r(2)$ are integers, \[
\left\lceil \frac{\max\{m_1n_1,m_2n_2\} -\left( mn-m_1n_1-m_2n_2\right)}{2}\right\rceil \le r(1)+r(2). \]
Furthermore, since one of $r(1)$ and $r(2)$ is at least $1$,\[
\max\left\{\left\lceil \frac{\max\{m_1n_1,m_2n_2\} -\left( mn-m_1n_1-m_2n_2\right)}{2}\right\rceil,1\right\} \le r(1)+r(2).
\]
Thus, by (\ref{eq:S}),
\begin{align*}
|S| &\ge mn-\sum_{i=1}^2m_in_i + \max\left\{\left\lceil \frac{\max\{m_1n_1,m_2n_2\} -\left( mn-\sum_{i=1}^2m_in_i\right)}{2}\right\rceil,1\right\} \\
    &\ge \alpha(m,n).\end{align*}
%
%

%

\bigskip

Now we will show that there is a connected safe set with the size $\alpha(m,n)$.


Let $((m_1^*,m_2^*),(n_1^*,n_2^*))$ be an element of $P_2(m,n)$ that satisfies $\alpha(m,n)$, that is,
\[\alpha(m,n)= mn-\sum_{i=1}^2m_i^*n_i^* + \max\left\{\left\lceil \frac{\max\{m_1^*n_1^*,m_2^*n_2^*\} -\left( mn-\sum_{i=1}^2m_i^*n_i^*\right)}{2}\right\rceil,1\right\}.\]
In addition, we let
\[ D_1 =\{ (i,j) \in V(G)\mid  1\le i\le m_1^* \text{ and }   1\le j\le n_1^* \}\]
and
\[D_2 =\{ (i,j) \in V(G)\mid  m_1^*+1\le i\le m \text{ and }   n_1^*+1\le j\le n \}.\]
Then $|D_i| = m_i^*n_i^*$ for each $i=1,2$.

For simplicity, for each $ t \in \{1,2\}$,
let $\nu_t$ be a nonnegative integer such that
\[ \nu_t = \left\lceil \frac{m_t^*n_t^* -\left( mn-\sum_{i=1}^2m_i^*n_i^*\right)}{2}\right\rceil. \]
Then, by Lemma~\ref{component_size}, there is at most one $\nu_t$ such that $\nu_t \ge 1$. Without loss of generality, we may assume that $\nu_2 \le 0$.


Suppose that $\nu_1 \le 0$. Then we let
\[S=\left( V(G)\setminus \left(D_1 \cup  D_2\right) \right)\cup \{(1,1)\}.\]
By definition, it is clear that $S$ is connected and any component of $G-S$ is contained in $D_1$ or $D_2$. Furthermore, \begin{align*}
|S| &=  mn-\sum_{i=1}^2m_i^*n_i^* + 1 \\
    &=  mn-\sum_{i=1}^2m_i^*n_i^* + \max\left\{\left\lceil \frac{\max\{m_1^*n_1^*,m_2^*n_2^*\} -\left( mn-\sum_{i=1}^2m_i^*n_i^*\right)}{2}\right\rceil,1\right\} \\
    &=\alpha(m,n). \end{align*}
Suppose that $G-S$ has two components $C_1$ and $C_2$ such that $C_t \subset D_t$ for each $t\in \{1,2\}$.
Since $\nu_1\le 0 $, $m_1^*n_1^* \le mn-\sum_{i=1}^2m_i^*n_i^*$, it holds that
\[ |C_1| \le |D_1| = m_1^*n_1^* \le  mn-\sum_{i=1}^2m_i^*n_i^*\le \alpha(m,n) = |S|.\]
Similarly, since $\nu_2 \le 0$,
\[ |C_2|\le |D_2| = m_2^*n_2^* \le   mn-\sum_{i=1}^2m_i^*n_i^* \le \alpha(m,n)= |S|.\]
Therefore $S$ is a safe set.

Now, suppose that $\nu_1 \ge 1$. By the division algorithm, there are integers $q,r$ such that $\nu_1 = m_1q+r$ with $0 \le r < m_1$. Let \[
D_1' = \bigcup_{i=1}^{q}  \{(1,i), \ldots, (m_1,i)\} \cup \{ (1,q+1), \ldots, (r,q+1)\}\]
where $\bigcup_{i=1}^{q} \{(1,i), \ldots, (m_1,i)\} = \emptyset$ if $q=0$. Note that $D_1' \subset D_1$ and $|D_1'|=\nu_1$.
Let \[
S=\left( V(G)\setminus  \left(D_1 \cup  D_2\right)\right) \cup D'_1.\]
Then\[|S| =  mn-\sum_{i=1}^2m_i^*n_i^* + \nu_1.\]
Since $\nu_2 \le 0$ and $\nu_1 \ge 1$,
\[\nu_1=\max\left\{\left\lceil \frac{\max\{m_1^*n_1^*,m_2^*n_2^*\} -\left( mn-\sum_{i=1}^2m_i^*n_i^*\right)}{2}\right\rceil,1\right\}.\]
Thus
\[|S|=\alpha(m,n).\]
By its construction, it is clear that $S$ is connected, one component of $G-S$ is contained in $D_1\setminus D'_1$, and the other component is contained in $D_2$.
Suppose that $G-S$ has two components $C_1$ and $C_2$ such that $C_1 \subset D_1\setminus D'_1$ and $C_2 \subset D_2$.
Since $\nu_2 \le 0$,
\[ |C_2|\le |D_2| = m_2^*n_2^* \le  mn-\sum_{i=1}^2m_i^*n_i^* \le  |S|.\]
Now \begin{align*}
&|C_1| \le |D_1\setminus D'_1|
    =m_1^*n_1^* - \nu_1
    \le \left(m_1^*n_1^* -2\nu_1\right) + \nu_1 \\
    &\le \left[m_1^*n_1^* -2 \cdot \frac{m_1^*n_1^* -\left( mn-\sum_{i=1}^2m_i^*n_i^*\right)}{2}\right] + \nu_1=mn-\sum_{i=1}^2m_i^*n_i^*+\nu_1
    = |S|.\end{align*}
Therefore $S$ is a safe set.\end{proof}

\begin{Thm}\label{thm_equal}
Let $G=K_m \Box K_n$ for some $n\ge m \ge 1$. Then $s(G) = cs(G)$.\end{Thm}
\begin{proof} By Proposition~\ref{prop:simple}, it is sufficient to consider the cases $n \ge m \ge 3$. It is obvious that $s(G) \le cs(G)$. We show that $cs(G) \le s(G)$. By Proposition~\ref{prop:1comp}, there is a connected safe set of size $\left\lceil\frac{mn-1}{2}\right\rceil$ unless $m=n=3$. It is easy to check that $\{(1,1),(1,2),(1,3),(2,1),(2,2)\}$ is a connected safe set of $K_3 \Box K_3$. Therefore $cs(K_3 \Box K_3)\le 5$.
Thus $cs(G) \le \left\lceil \frac{mn}{2} \right\rceil$ for $n \ge m \ge 3$.

Let $S$ be a minimum safe set of $G$. By Proposition~\ref{prop:1comp}, we may assume that $S$ is a vertex cut. If $G[S]$ is connected, then we are done. Now suppose that $S$ is not connected, that is, $\omega(G[S]) =t \ge 2$ for some nonnegative integer $t$. Then, by Lemma~\ref{lem_connected}, $\omega(G-S) \le 2$.
As we have shown that $cs(G) \le \left\lceil \frac{mn}{2} \right\rceil$, it is sufficient to show that $\left\lceil \frac{mn}{2} \right\rceil \le |S|$.
 Suppose that $\omega(G-S)=1$. Since $S$ is a safe set, $|G-S| \le |S|$. Then, since $|V(G)| = |G-S| +|S|$, $|V(G)| \le 2|S|$ or $\left\lceil\frac{mn}{2}\right\rceil \le |S|$.

Now suppose that $\omega(G-S)=2$. Then, by Lemma~\ref{lem_connected}, $|\Pi_1(1)| + |\Pi_1(2)| = m$, and $|\Pi_2(1)| + |\Pi_2(2)| = n$ where $(\Pi_1,\Pi_2)$ is the component projection induced by $S$.  Let $C_1$ and $C_2$ be the components of $G-S$. Then $G[S]$ has two components $S_1$ and $S_2$, and $|C_1|=|\Pi_1(1)||\Pi_2(1)|$, $|C_2|=|\Pi_1(2)||\Pi_2(2)|$,   $|S_1|=|\Pi_1(1)||\Pi_2(2)|$, $|S_2|=|\Pi_1(2)||\Pi_2(1)|$. See Figure~\ref{fig:simple} for an illustration. Moreover, there are edges joining a vertex in $C_1$ and a vertex in $S_1$, a vertex in $C_1$ and a vertex in $S_2$, a vertex in $C_2$ and a vertex in $S_1$, a vertex in $C_2$ and a vertex in $S_2$, respectively.

Therefore, by the definition of a safe set,
\[|S_1| \ge \max\{|C_1|, |C_2|\} \quad \text{and} \quad |S_2| \ge \max\{|C_1|, |C_2|\}. \] Then \begin{eqnarray*}
mn &=& (|\Pi_1(1)| + |\Pi_1(2)|)(|\Pi_2(1)| + |\Pi_2(2)|)\\
	&=& |\Pi_1(1)||\Pi_2(1)| + |\Pi_1(1)||\Pi_2(2)| + |\Pi_1(2)||\Pi_2(1)| + |\Pi_1(2)||\Pi_2(2)|\\
	&=& |C_1| + |S_1| +  |S_2| + |C_2| \\
	&\le& 2(|S_1|+|S_2|) \le 2|S|,\end{eqnarray*}
so $\frac{mn}{2} \le |S|$. Since $|S|$ is an integer, $\lceil \frac{mn}{2}\rceil \le |S|$ and we complete the proof.\end{proof}

From Theorem~\ref{thm_min_connected} and Theorem~\ref{thm_equal}, we immediately obtain our main result.
\begin{Thm}\label{thm:main}
For two integers $m\ge n\ge 1$,
\[s(K_m \Box K_n)=cs(K_m \Box K_n)=\alpha(m,n).\]
\end{Thm}

The following is an algorithm for MATLAB computing $\alpha(m,n)$ in a polynomial time.

\begin{center}
\includegraphics[width=0.8\textwidth]{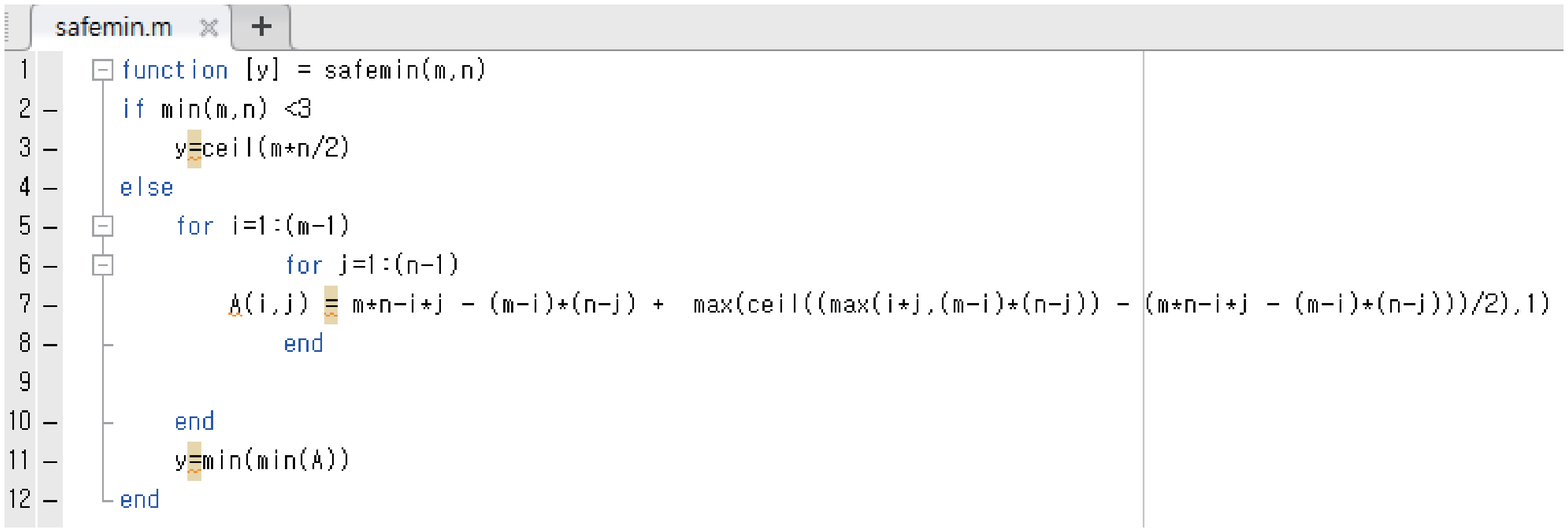}
\end{center}

The following are the values of $\alpha(m,n)$ for $n,m \le 10$ generated by the above algorithm.

\begin{center}
\includegraphics[width=0.7\textwidth]{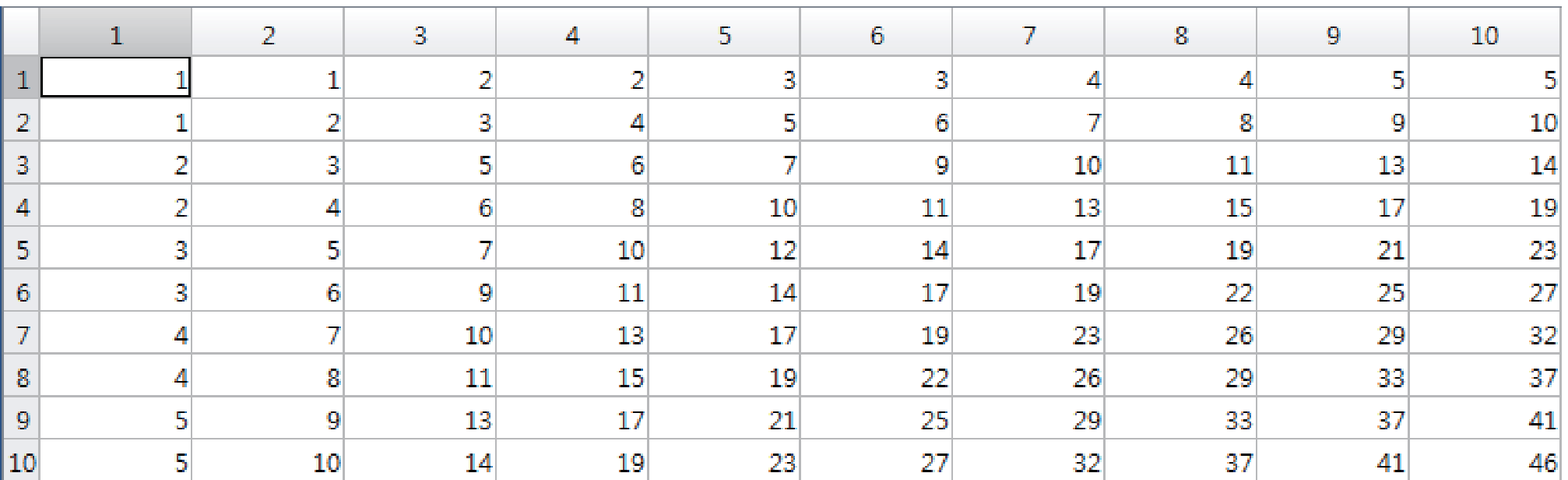}
\end{center}

\section{$\alpha(m,n)$ for $m \in \{3,4\}$, $n \ge m$ }

    For $m \in \{3,4\}$, $n \ge m$, we precisely formulate $\alpha(m,n)$.

\begin{Thm} The safe number of $K_m \Box K_n$ for $m \in \{3,4\}$, $n \ge m$ is as follows: \begin{itemize}
\item[(i)] If $m=3$, then \[
s(K_m \Box K_n)= n+ \left\lfloor\frac{n}{3} \right\rfloor +1;\]
\item[(ii)] If $m=4$, then\[
s(K_m \Box K_n) = n+ 4\cdot \left\lfloor\frac{n}{5}\right\rfloor +\max\{i,1\}\]
where $i \equiv n \pmod{5}$ for some $i$, $0 \le i \le 4$.
\end{itemize} \label{thm:precise34}
\end{Thm}

\begin{proof}
We first show that upper bounds of $s(K_m \Box K_n)$ are $n+ \left\lfloor\frac{n}{3} \right\rfloor +1$ and $n+ 4\cdot \left\lfloor\frac{n}{5}\right\rfloor +\max\{i,1\}$, respectively, for $m=3$ and $m=4$ where $i \equiv n \pmod{5}$ for some $i$, $0 \le i \le 4$.

Take $\left(\left(m_1,m_2\right),\left(n_1,n_2\right)\right) = \left(\left(1,2\right),\left(\left\lfloor\frac{n}{3} \right\rfloor, n- \left\lfloor\frac{n}{3} \right\rfloor\right)\right)$ from $P_2(3,n)$. Then\[
\sum_{i=1}^2 m_in_i = \left(\left\lfloor\frac{n}{3}\right\rfloor + 2\left(n- \left\lfloor\frac{n}{3}\right\rfloor \right) \right) = 2n - \left\lfloor\frac{n}{3}\right\rfloor,\]
\[
\max\{m_1n_1,m_2n_2\} = 2\left(n- \left\lfloor\frac{n}{3}\right\rfloor \right),\]
and so
\[
\Omega_3:=\max\left\{\left\lceil\frac{\max\{n_1,2n_2\} - n-n_1}{2}\right\rceil ,1 \right\} =\max\left\{\left\lceil\frac{2\left(n- \left\lfloor\frac{n}{3}\right\rfloor \right) - n - \left\lfloor\frac{n}{3}\right\rfloor}{2}\right\rceil ,1 \right\}=1.
\]
Therefore
\[\alpha(3,n) =3n-\sum_{i=1}^2m_in_i +\Omega_3 = n + \left\lfloor\frac{n}{3}\right\rfloor + 1.\]
 Hence\[
 s(K_3 \Box K_n) = \alpha(3,n) \le n + \left\lfloor\frac{n}{3}\right\rfloor + 1\]
  by Theorem~\ref{thm:main}.

Now we take $\left(\left(m_1,m_2\right),\left(n_1,n_2\right)\right) = \left(\left(1,3\right),\left(2\cdot \left\lfloor\frac{n}{5} \right\rfloor, n- 2\cdot\left\lfloor\frac{n}{5} \right\rfloor\right)\right)$ from $P_2(4,n)$. Suppose that $n=5q + i$ for some $i$, $0 \le i \le 4$. Then\[
\sum_{i=1}^2 m_in_i = 2\cdot \left\lfloor\frac{n}{5} \right\rfloor + 3\left(n- 2\cdot\left\lfloor\frac{n}{5} \right\rfloor\right) = 3n - 4\cdot\left\lfloor\frac{n}{5} \right\rfloor,\]
\[
\max\{m_1n_1,m_2n_2\} = 3\left(n- 2\cdot\left\lfloor\frac{n}{5} \right\rfloor\right),\]
and so
\begin{align*}
\Omega_4&:=\max\left\{\left\lceil\frac{\max\{m_1n_1,m_2n_2\} - 4n+\sum_{i=1}^2m_in_i}{2}\right\rceil ,1 \right\} \\ &=\max\left\{\left\lceil\frac{3\left(n- 2\cdot\left\lfloor\frac{n}{5} \right\rfloor\right) - n - 4\cdot\left\lfloor\frac{n}{5} \right\rfloor}{2}\right\rceil,1\right\} =\max\{i,1\}.
\end{align*}
Therefore \[
\alpha(4,n) = 4n-\sum_{i=1}^2m_in_i +\Omega_4  = n + 4\cdot \left\lfloor\frac{n}{5}\right\rfloor + \max\{i,1\}.\]
 Hence\[
s(K_4 \Box K_n) = \alpha(4,n) \le n + 4\cdot \left\lfloor\frac{n}{5}\right\rfloor + \max\{i,1\}\]
 by Theorem~\ref{thm:main}.

 Now we show that $\alpha(3,n) \ge n+ \left\lfloor\frac{n}{3} \right\rfloor +1$ for $n \ge 4$.
 Let $\left((m_1,m_2),(n_1,n_2)\right) \in P_2(3,n)$ be a partition by which $\alpha(3,n)$ is achieved. That is, \[
 \alpha(3,n) = 3n-\sum_{i=1}^2m_in_i + \max\left\{\left\lceil\frac{\max\{n_1,2n_2\} - n-n_1}{2}\right\rceil ,1 \right\}.  \]

  Without loss of generality, we may assume that $m_1 = 1$ and $m_2 = 2$.
 Then
 \[3n-\sum_{i=1}^2m_in_i=3n - n_1 - 2n_2=n+n_1,\]
so \[\Omega_3:=\max\left\{\left\lceil\frac{\max\{n_1,2n_2\} - n-n_1}{2}\right\rceil ,1 \right\} \quad \text{and} \quad
\alpha(3,n) = n+n_1 +\Omega_3.\]

Suppose $n_1 \ge 2n_2$. Then $\max\{n_1,2n_2\}-n-n_1=-n<0$, so $\Omega_3=1$. Thus $\alpha(3,n)=n+n_1+1$. On the other hand, since $n_1+n_2=n$, $n_1 \ge 2n_2$ implies
$2n_2 \ge n$. Then
\[\alpha(3,n)=n+n_1+1 \ge n+\frac{n}{2}+1 > n+\left\lfloor \frac{n}{3} \right\rfloor+1,\]
and we reach a contradiction as we have shown that $\alpha(3,n) \le n+\left\lfloor \frac{n}{3} \right\rfloor+1$. Thus $n_1 < 2n_2$ and so $\max\{n_1,2n_2\}=2n_2$. Then
\[ 2n_2-n-n_1  = 2(n-n_1) - n -n_1
                         = n - 3n_1,\]
so $\Omega_3=\max\left\{\left\lceil \frac{n-3n_1}{2}\right\rceil,1\right\}$.

Suppose that $n-3n_1<2$. Then $n_1>\frac{n-2}{3}$ and $\Omega_3=1$. Since $n_1$ is an integer, $n_1>\left\lfloor \frac{n}{3} \right\rfloor$. Therefore
\[
 \alpha(3,n)= n+n_1 +1 > n + \left\lfloor\frac{n}{3} \right\rfloor + 1\]
and we reach a contradiction. Therefore $n-3n_1 \ge 2$ and so $\Omega_3=\left\lceil \frac{n-3n_1}{2}\right\rceil$. Thus
\[
  \alpha(3,n)
  = n+n_1+ \left\lceil \frac{n-3n_1}{2}\right\rceil = n+\left\lceil\frac{n-n_1}{2}\right\rceil
  \ge n + \left\lceil\frac{n+1}{3}\right\rceil. \]
Since $\left\lceil\frac{n+1}{3}\right\rceil \ge \left\lfloor\frac{n}{3} \right\rfloor +1$ for $n \ge 4$, we obtain  $\alpha(3,n) \ge n+ \left\lfloor\frac{n}{3} \right\rfloor +1$ from the above inequality. Since $\alpha(3,n) \le n+ \left\lfloor\frac{n}{3} \right\rfloor +1$, we conclude that $\alpha(3,n) = n+ \left\lfloor\frac{n}{3} \right\rfloor +1$.

In the following, we will show that $n+ 4\cdot \left\lfloor\frac{n}{5}\right\rfloor +\max\{i,1\} \le \alpha(4,n)$ where $i \equiv n \pmod{5}$ for some $i$, $0 \le i \le 4$.

Let $\left((m_1,m_2),(n_1,n_2)\right) \in P_2(4,n)$ be a partition by which $\alpha(4,n)$ is achieved, that is,
\[
\alpha(4,n) =4n-\sum_{i=1}^2m_in_i + \max\left\{\left\lceil\frac{\max\{m_1n_1,m_2n_2\}-n-2n_1}{2}\right\rceil ,1 \right\}. \]
Then either $\{m_1,m_2\} = \{2\}$ or $\{m_1,m_2\} = \{1,3\}$.
Assume $\{m_1,m_2\} = \{2\}$. Then \[
4n-\sum_{i=1}^2m_in_i=4n - 2n_1 - 2n_2=2n.\]
 Since $ \max\left\{\left\lceil\frac{\max\{2n_1,2n_2\}-n-2n_1}{2}\right\rceil ,1 \right\} \ge 1$, $\alpha(4,n) \ge 2n+1$. However, we have already shown that $\alpha(4,n) \le n+ 4\cdot \left\lfloor\frac{n}{5}\right\rfloor +\max\{i,1\}$, which is less than $2n+1$, and we reach a contradiction. Thus $\{m_1,m_2\} = \{1,3\}$.
 Without loss of generality, we may assume that $m_1=1$ and $m_2 = 3$. Then\[
4n-\sum_{i=1}^2m_in_i=4n - n_1 - 3n_2=n+2n_1\]
and  \[\Omega_4:= \max\left\{\left\lceil\frac{\max\{n_1,3n_2\}-n-2n_1}{2}\right\rceil ,1 \right\}.\]
If $n_1 \ge 3n_2$, then $\max\{n_1,3n_2\}-n-2n_1  = -n-n_1<0$, a contradiction. Therefore $3n_2 \ge n_1$ and so $\Omega_4= \max\left\{\left\lceil\frac{2n-5n_1}{2}\right\rceil ,1 \right\}$ since $n_1+n_2=n$.
If $2n-5n_1<2$, then $\Omega_4=1$ and so
\[\alpha(4,n)=n+2n_1+1> n+4\cdot \left\lfloor\frac{n}{5}\right\rfloor + \max\{i,1\},\]
which is a contradiction. Thus $2n-5n_1 \ge 2$. Then $n_1 \le \frac{2n-2}{5}$ and $\Omega_4=\left\lceil\frac{2n-5n_1}{2}\right\rceil$, so
\[\alpha(4,n)=n+2n_1+\left\lceil\frac{2n-5n_1}{2}\right\rceil=n+\left\lceil\frac{2n-n_1}{2}\right\rceil \ge n+ \left\lceil\frac{1+4n}{5}\right\rceil.\]
Since $i \equiv n \pmod{5}$, $n=5k+i$ for some integer $k$. Thus
\[\left\lceil\frac{1+4n}{5}\right\rceil = \left\lceil\frac{1+20k+4i}{5}\right\rceil
                                      = 4k + \left\lceil\frac{1+4i}{5}\right\rceil
                                      \ge 4\cdot \left\lfloor\frac{n}{5}\right\rfloor +  \max\{i,1\}\]
and so $\alpha(4,n) \ge n+ 4\cdot \left\lfloor\frac{n}{5}\right\rfloor +\max\{i,1\}$.
Thus $\alpha(4,n) = n+ 4\cdot \left\lfloor\frac{n}{5}\right\rfloor +\max\{i,1\}$
and we complete the proof.\end{proof}

\end{document}